\documentclass[12pt]{article}

\usepackage{amssymb}
\usepackage{amsmath}
\usepackage{amsfonts}
\usepackage{times}
\usepackage{graphicx}
\usepackage{amsthm}

\def \RR {\mathbb R}

\def \EE {\mathbb E}
\def \ZZ {\mathbb Z}
\def \CC {\mathbb C}
\def \PP {\mathbb P}
\def \eps {\varepsilon}
\def \vphi {\varphi}

\def \cF {\mathcal F}

\def \cN {\mathcal N}

\def \cS {\mathcal S}
\def \cJ {\mathcal J}

\newtheorem{theorem}{Theorem}[section]
\newtheorem{lemma}[theorem]{Lemma}

\newtheorem{corollary}[theorem]{Corollary}
\theoremstyle{definition}

\def\myffrac#1#2 in #3{\raise 2.6pt\hbox{$#3 #1$}\mkern-1.5mu\raise 0.8pt\hbox{$
#3/$}\mkern-1.1mu\lower 1.5pt\hbox{$#3 #2$}}

\begin{document}

\title{Variations on the Berry-Esseen theorem}
\author{Bo'az Klartag\textsuperscript{1} and Sasha
Sodin\textsuperscript{2}}

\footnotetext[1]{Supported in part by the Israel Science Foundation
and by a Marie Curie Reintegration Grant from the Commission of the
European Communities.} \footnotetext[2]{Supported in part by the
Adams Fellowship Program of the Israel Academy of Sciences and
Humanities and by the Israel Science Foundation.}

\date{}
\maketitle \abstract{ Suppose that $X_1,\ldots,X_n$ are independent,
identically-distributed random variables of mean zero and variance
one. Assume that $\EE |X_1|^4 \leq \delta^4$. We observe that there
exist many choices of coefficients $\theta_1,\ldots,\theta_n \in
\RR$ with $\sum_j \theta_j^2 = 1$ for which
 \begin{equation}
 \sup_{\alpha, \beta \in \RR \atop{\alpha < \beta}}
 \left| \PP \left( \alpha \leq \sum_{j=1}^n \theta_j X_j \leq \beta \right) \, -
\, \frac{1}{\sqrt{2 \pi}} \int_{\alpha}^{\beta} e^{-t^2/2}dt \right|
\leq \frac{C \delta^4}{n}, \label{eq_1417}
\end{equation}
where  $C> 0$ is a universal constant.
 Inequality (\ref{eq_1417})
should be compared with the classical
 Berry-Esseen theorem, according to which
 the left-hand side of (\ref{eq_1417}) may decay with $n$ at the slower rate of $O( 1/
 \sqrt{n})$,
 for the unit vector $\theta = (1,\ldots,1) / \sqrt{n}$.  An explicit, universal example for coefficients $\theta = (\theta_1,\ldots,\theta_n)$ for which
(\ref{eq_1417}) holds is
\[ \theta = (1,
\sqrt{2}, -1, -\sqrt{2}, 1, \sqrt{2}, -1, -\sqrt{2}, \cdots)
    \, \Big/ \, \sqrt{3n/2} ~ \]
when $n$ is divisible by four. Parts of the argument are applicable
also in the more general case, in which $X_1,\ldots,X_n$ are
independent random variables of mean zero and variance one, yet they
are not necessarily identically distributed. In this general
setting, the bound (\ref{eq_1417}) holds with $\delta^4 = n^{-1}
\sum_{j=1}^n \EE |X_j|^4$ for most selections of a unit vector
$\theta = (\theta_1,\ldots,\theta_n) \in \RR^n$. Here ``most''
refers to the uniform probability measure
 on the unit sphere. }

\section{Introduction}
\label{sec1}

This note brings further evidence for the fundamental r\^ole played
by the geometry of the high-dimensional sphere in the analysis of
the central limit theorem, in the spirit of works by Sudakov
\cite{S}, Diaconis and Freedman \cite{DF} and others. Suppose that
$X_1,\ldots,X_n$ are independent random variables with finite third
moments such that $\EE X_j = 0$ and $\EE X_j^2 = 1$ for all $j$. The
classical Berry-Esseen theorem (see, e.g., Feller \cite[Vol. II,
Chapter XVI]{feller}) states that
\begin{equation}
 \sup_{\alpha, \beta \in \RR \atop{\alpha < \beta}} \left| \PP
\left( \alpha \leq \frac{1}{\sqrt{n}} \sum_{j=1}^n X_j  \leq \beta
\right) \, - \, \frac{1}{\sqrt{2 \pi}} \int_{\alpha}^{\beta}
e^{-t^2/2}dt \right| \leq \frac{C \gamma^3}{\sqrt{n}}
\label{eq_1445}
\end{equation}
where $\gamma = \left( \sum_j \EE |X_j|^3 / n \right)^{1/3} \leq
\max_j (\EE |X_j|^3)^{1/3}$ and $C > 0$ is a universal constant. In
the general case, where $X_1,\ldots,X_n$ are
 non-symmetric random variables, the bound
(\ref{eq_1445}) is sharp. Even when $X_1,\ldots,X_n$ are symmetric
random variables, the bound (\ref{eq_1445}) may not be improved in
general: If the random variables are symmetric Bernoulli variables,
for instance, then the probability $\PP(\sum_j X_j = 0)$ is
approximately $\left( \pi n / 2 \right)^{-1/2}$ for large even $n$.
Therefore (\ref{eq_1445}) is an asymptotically optimal bound in this
case, up to the value of the constant $C$.

\medskip Quite unexpectedly, we find that there exists a linear combination
of the random variables $X_1,\ldots,X_n$ that is much closer to the
standard gaussian distribution.  As it
turns out, selecting the coefficients of the linear combination in a
probabilistic fashion may significantly improve the rate of
convergence to the gaussian distribution. We denote  $S^{n-1} = \{
(x_1,\ldots,x_n) \in \RR^n ; \sum_i x_i^2 = 1 \}$, the unit sphere
in $\RR^n$. Let $\sigma_{n-1}$ be the unique rotationally-invariant
probability measure on $S^{n-1}$, referred to as the uniform
distribution on the sphere. Whenever we say that a random vector is
distributed uniformly on the sphere, we mean that it is distributed
according to $\sigma_{n-1}$. The coefficients of the linear
combination will be selected randomly, uniformly over the sphere.

\begin{theorem} Let $n \geq 1$ be an integer, $0 < \rho < 1$.
Suppose that $X_1,\ldots,X_n$ are independent random variables with
finite fourth moments, such that $\EE X_j =0$ and $\EE X_j^2 = 1$ for $j=1,\ldots,n$.
Denote
$$ \delta = \left( \frac{1}{n} \sum_{j=1}^n \EE X_j^4
\right)^{1/4}. $$  Then, there exists a subset $\cF \subseteq
S^{n-1}$ with $\sigma_{n-1}(\cF) \geq 1 - \rho$ for which the
following holds: For any $\theta = (\theta_1,\ldots,\theta_n) \in
\cF$, \begin{equation} \sup_{\alpha, \beta \in \RR \atop{\alpha <
\beta}} \left| \PP \left( \alpha \leq \sum_{j=1}^n \theta_j X_j \leq
\beta \right) \, - \, \frac{1}{\sqrt{2 \pi}} \int_{\alpha}^{\beta}
e^{-t^2/2}dt \right| \leq \frac{C(\rho) \delta^4}{n} \label{eq_1412}
\end{equation} where $C(\rho)$ is a constant depending solely on
$\rho$. In fact, $C(\rho) \leq C \log^2 \left( 1/\rho \right)$,
where $C
> 0$ is a universal constant. \label{thm_1517}
\end{theorem}

 A case of interest is when $X_1,X_2,\ldots$ is an infinite
sequence of independent, identically-distributed random variables of
mean zero and variance one, with finite fourth moment. In this
case, Theorem \ref{thm_1517} provides a convergence to the gaussian
distribution, with rate of convergence of the order $O(1 / n)$, for
appropriate (or random) choice of linear combinations.

\medskip The case where $X_1,\ldots,X_n$ are
identically-distributed, independent random variables is quite
remarkable here. In this case the subset $\cF \subseteq S^{n-1}$
from Theorem \ref{thm_1517} can be described explicitly, and it does
not depend on the distribution of the random variables. Following
Rudelson and Vershynin \cite{RV}, we make use of arithmetic
properties of the vector $\theta$. Define
\[ d(x, \ZZ) = \min_{p \in \ZZ} |p - x|~, \quad \quad (x \in \RR)~, \]
and for $\theta = (\theta_1,\ldots,\theta_n) \in \RR^n$ set,
\[ d(\theta, \ZZ^n) = \sqrt{\sum_{j=1}^n d^2(\theta_j, \ZZ)}. \]
Given $\theta \in S^{n-1}$ we denote by $\cN(\theta)$ the minimal  $R \geq 1$ for which the following three conditions
hold:
\begin{enumerate}
\item[(i)] $\displaystyle \left| \sum_{j=1}^n \theta_j^3 \right| \leq R /
n$.
\item[(ii)] $\displaystyle  \sum_{j=1}^n \theta_j^4  \leq R / n$.
\item[(iii)] For any $|\xi| \leq n$,
$$
 d(\xi \theta, \ZZ^n) \geq \frac{1}{10} \min \left \{ |\xi|,  \frac{n / R}{|\xi|} \right \}. $$
(the number $10$ does not play any special r\^ole)
\end{enumerate}
 The three conditions above are satisfied by
many unit vectors in the unit sphere, and are not very difficult to
verify in certain examples. In order to appreciate the third condition, observe
that for a typical unit vector $\theta \in S^{n-1}$ we have
$$ d(\xi \theta, \ZZ^n) \geq \min \{ |\xi|, c \sqrt{n}  \} \quad \quad \text{for any} \ |\xi| \leq e^{c n}, $$ where $c > 0$ is a universal constant. For a concrete example,  consider the
unit vector
\begin{equation}
 \theta^0 = (1,
\sqrt{2}, -1, -\sqrt{2}, 1, \sqrt{2}, -1, -\sqrt{2}, \cdots)
    \, \Big/ \, \sqrt{3n/2}~ \label{eq_1412_}
\end{equation}
for $n$ divisible by four.  For this unit vector, the sum in
 (i) is zero, whereas (ii) clearly holds for any $R \geq 2$. The
third condition is verified in Lemma~\ref{l:sqrt2}, hence
$$ \cN(\theta^0) \leq C $$ for a universal constant $C > 0$.

\begin{theorem}\label{thm_1328}
 Suppose that $X_1,\ldots,X_n$ are independent,
identically-distributed  random variables with finite fourth
moments, such that $$ \EE X_1 =0, \ \ \ \EE X_1^2 = 1 \ \ \ \ \
\text{and} \ \ \ \ \EE X_1^4 \leq \delta^4. $$ Then, for any $\theta
\in S^{n-1}$,
\begin{equation}
 \sup_{\alpha, \beta \in \RR \atop{\alpha < \beta}} \left| \PP
\left( \alpha \leq \sum_{j=1}^n \theta_j X_j \leq \beta \right) \, -
\, \frac{1}{\sqrt{2 \pi}} \int_{\alpha}^{\beta} e^{-t^2/2}dt \right|
\leq  \frac{C \cN(\theta) \delta^4}{n}, \label{eq_908}
\end{equation} where $C
> 0$ is a universal constant.
\end{theorem}

Once formulated, Theorem \ref{thm_1517} and Theorem \ref{thm_1328}
require nothing but an adaptation of the proofs of the classical
quantitative bounds in the central limit theorem. The following
pages contain the details of the argument. Section \ref{sec2} serves
mostly as a remainder for the proof of the Berry-Esseen bound using
the Fourier transform. In Section \ref{sec3} and Section \ref{sec4}
we exploit the randomness involved in the selection of
$\theta_1,\ldots,\theta_n$ in Theorem \ref{thm_1517}. Section
\ref{sec5} is devoted to the proof of Theorem \ref{thm_1328}.

\medskip Throughout this text, the letters $c, \tilde{c}, c^{\prime},
C, \tilde{C}, \bar{C}$ etc. stand for various positive universal
constants, whose value may change from one line to the next. We
usually use upper-case $C$ to denote universal constants that we
think of as ``sufficiently large'', and lower-case $c$ to denote
universal constants that are ``sufficiently small''. The notation
$O(x)$, for some expression $x$, is an abbreviation for some
complicated quantity $y$ with the property that $ |y| \leq C x $ for
some universal constant $C > 0$. A standard Gaussian random variable
is a random variable whose density is $t \mapsto (2 \pi)^{-1/2}
\exp(-t^2/2)$ on the real line.

\medskip
{\bf Acknowledgement.} We would like to thank Shahar Mendelson for
his help on the subject of  Bernstein-type inequalities in the
absence of exponential moments.

\section{The Fourier inversion formula}
\label{sec2}

Throughout this note, $X_1,\ldots,X_n$ are
independent random variables with finite fourth moments such that
$\EE X_j = 0$ and $\EE X_j^2 = 1$ for all $j$. Denote
$$ \gamma_j = \left( \EE X_j^3 \right)^{1/3}, \ \ \ \
\bar{\gamma}_j =  \left( \EE |X_j|^3 \right)^{1/3}, \ \ \ \ \delta_j
= \left( \EE X_j^4 \right)^{1/4} \ \ \ \ \ \ \ \ \ \ \ (1 \leq j
\leq n). $$ Note that $\gamma_j$ may be negative, since $\EE X_j^3$
does not have a definite sign, and the third root of a negative
number is negative. To help the reader remember the r\^{o}les of the
Greek letters, we confess right away that $\gamma$, the third letter
in the Greek alphabet, represents third moments, while $\delta$, the
fourth letter in the Greek alphabet, represents fourth moments. We
also set
$$ \delta = \left( \frac{1}{n} \sum_{j=1}^n \delta_j^4
\right)^{1/4}.
$$
According to the Cauchy-Schwartz inequality, $\bar{\gamma}_j^3 = \EE
|X_j|^3 \leq \sqrt{\EE X_j^4 \EE X_j^2} = \delta_j^2$. Hence,
\begin{equation}
 |\gamma_j| \leq  \bar{\gamma}_j \leq \delta_j^{2/3}, \ \bar{\gamma}_j\geq 1, \ \delta_j \geq 1 \ \ \ \ \ \ \ \ \ \ \ \text{for} \ j=1,\ldots,n.
\label{eq_1633}
\end{equation}
 Consider the Fourier transform
$$ \vphi_j(\xi) = \EE \exp \left(-i \xi X_j \right), \ \ \ \ \ \ \ \ \ (\xi \in \RR, 1 \leq j \leq n) $$
where $i^2 = -1$. Clearly $|\vphi_j(\xi)| \leq 1$ for any $\xi \in
\RR$. The $k^{th}$ derivative of $\vphi_j$ is
$$ \vphi_j^{(k)}(\xi) =
(-i)^k \EE X_j^k \exp(-i \xi X_j) $$ for any $\xi \in \RR, 1 \leq j
\leq n$ and $0 \leq k \leq 4$. Consequently,
\begin{equation}
 \vphi_j(0) = 1, \ \ \vphi_j^{\prime}(0) = 0, \ \ \vphi_j^{\prime \prime}(0) = -1, \ \
\vphi_j^{(3)}(0) = i \gamma_j^3 \label{eq_942}
\end{equation}
for all $j$, and
\begin{equation}
|\vphi_j^{(3)}(\xi)| \leq \bar{\gamma}_j^3, \ \ \
|\vphi_j^{(4)}(\xi)| \leq \delta_j^4 \ \ \ \ \ \ \ \ \ \ \ \text{for
all} \ \xi \in \RR, 1 \leq j \leq n. \label{eq_943_}
\end{equation}
For a unit vector $\theta = (\theta_1,\ldots,\theta_n) \in S^{n-1}$,
by independence,
$$ \vphi_{\theta}(\xi) := \EE \exp \left(-i \xi \sum_{j=1}^n \theta_j X_j \right) =
\prod_{i=1}^n \EE \exp \left(-i \xi \theta_j X_j \right) =
\prod_{j=1}^n \vphi_j( \theta_j \xi). $$ Denote, for $\theta \in
S^{n-1}$,
$$ F_{\theta}(t) = \PP \left( \sum_{j=1}^n \theta_j X_j \leq  t \right), \ \
 \ \ \ \ \Phi(t) = \frac{1}{\sqrt{2 \pi}} \int_{-\infty}^t e^{-s^2/2} ds \ \ \ \ \ \ \ \ \ (t \in \RR). $$
 Recall that when
$\Gamma$ is a standard Gaussian random variable, $\EE \exp(-i \xi
\Gamma) = \exp(-\xi^2/2)$. In order to control $\sup_t
|F_{\theta}(t) - \Phi(t)|$ it is customary to try and bound the
difference of the Fourier transforms $|\vphi_{\theta}(\xi) -
\exp(-\xi^2/2)|$ for $\xi$ in a large enough interval. According to
Lemma 2 in \cite[Vol. II, Section XVI.3]{feller}, whose proof is
based on a simple smoothing technique,
\begin{equation}
 \sup_{t \in \RR} |F_{\theta}(t) - \Phi(t)| \leq C \int_{-T}^T \frac{|\vphi_{\theta}(\xi) -
\exp(-\xi^2/2)|}{|\xi|} d \xi + \frac{C}{T}, \label{eq_958}
\end{equation}
for any $T > 0$, where $C > 0$ is a universal constant.

\medskip It is important to mention that for large classes of
probability distributions, the error term $C / T$ in (\ref{eq_958})
is non-optimal, and may be improved upon to $C / T^2$ in some cases.
See, for instance, Lemma 9 in \cite{ptrf} for an improvement of this
nature pertaining to even, log-concave distributions. We are
dealing, however, with arbitrary random variables, hence we must
rely on the bound (\ref{eq_958}). Thus, in order to prove Theorem
\ref{thm_1517}, we need to establish
\begin{equation}
\int_{-n/\delta^4}^{n/\delta^4} \frac{|\vphi_{\theta}(\xi) -
\exp(-\xi^2/2)|}{|\xi|} d \xi \leq \frac{ C(\rho) \delta^4}{n}
\label{eq_1617}
\end{equation}
for all $\theta \in \cF$ where $\cF$ is a certain subset of the
sphere with $\sigma_{n-1}(\cF) \geq 1- \rho$. The rest of this paper
is devoted to the proof of (\ref{eq_1617}) and of the analogous
inequality in the context of Theorem \ref{thm_1328}. We divide the
domain of integration in (\ref{eq_1617}) into three parts. The
contribution of two of these domains is analyzed in the following
two lemmas.

\begin{lemma} Let $\theta = (\theta_1,\ldots,\theta_n) \in S^{n-1}$. Denote $\eps = \left( \sum_{j=1}^n \theta_j^4 \delta_j^4 \right)^{1/4}$
and $R_1 = \left| \sum_{j=1}^n \gamma_j^3 \theta_j^3 \right|$.
Suppose that $\eps \leq 1$. Then,
$$ \int_{-\eps^{-2/3}}^{\eps^{-2/3}} \left| \vphi_{\theta}(\xi) -
e^{-\xi^2/2} \right| \frac{d \xi}{|\xi|} \leq C \left[ R_1 + \eps^4 \right]
$$
where $C >0$ is a universal constant. \label{lem_dom1}
\end{lemma}

\begin{proof} Recall (\ref{eq_942}) and (\ref{eq_943_}). Taylor's theorem implies that for any $j=1,\ldots,n$ and $s \in \RR$,
$$ \left| \vphi_j(s) - \left[ 1 - \frac{1}{2} s^2 + \frac{i \gamma_j^3}{6} s^3 \right] \right|
\leq \frac{\delta_j^4}{24} s^4. $$ Recall that $\max \{ 1,
|\gamma_j| \} \leq \delta_j^{2/3} \leq \delta_j$ according to (\ref{eq_1633}). In
particular,
$$ |\vphi_j(s) - 1| \leq 3/4 \ \ \ \ \ \ \ \  \ \ \ \ \ \ \text{for } \ |s| \leq \delta_j^{-1}, \ j=1,\ldots,n. $$
Thus $\log \vphi_j(s)$ is well-defined for $|s| \leq \delta_j^{-1}$,
and for any $|s| \leq \delta_j^{-1}$ and $j=1,\ldots,n$,
\begin{equation}
 \log \vphi_j(s) = -\frac{1}{2} s^2 + \frac{i \gamma_j^3}{6} s^3 + O(\delta_j^4 s^4)
\label{eq_906}
\end{equation}
since  $| \log (1 + z) - z | \leq 8 |z|^2$ whenever  $|z| \leq 3/4$.
 Note that
for any $\xi \in \RR$ with $|\xi| \leq \eps^{-2/3}$,
$$ |\theta_j \xi| \leq |\theta_j| \eps^{-2/3} \leq |\theta_j| \eps^{-1} \leq \delta_j^{-1}. $$
Summing (\ref{eq_906}) over $j=1,\ldots,n$, we conclude that for any
$|\xi| \leq \eps^{-2/3}$,
\begin{equation}
 \sum_{j=1}^n \log \vphi_j(\theta_j \xi) =
-\frac{\xi^2}{2} + \frac{i \sum_{j=1}^n \gamma_j^3 \theta_j^3 }{6}
\xi^3 + O \left( \eps^4 \xi^4 \right) \label{eq_915}
\end{equation}
as $\sum_j  \theta_j^2 = 1$ and $\eps^4 = \sum_{j=1}^n \delta_j^4
\theta_j^4$. Recall that $|\gamma_j|^3 \leq \delta_j^2$. By the
Cauchy-Schwartz inequality,
\begin{equation}
 R_1 = \left| \sum_{j=1}^n \gamma_j^3 \theta_j^3
 \right| \leq \sum_{j=1}^n \delta_j^2 |\theta_j| \theta_j^2
 \leq \left( \sum_{j=1}^n \delta_j^4 |\theta_j|^2 \theta_j^2
 \right)^{1/2} = \eps^2. \label{eq_2210}
\end{equation}
Hence $R_1 |\xi|^3 + \eps^4 \xi^4
\leq 2$ for all $|\xi| \leq \eps^{-2/3}$.
From
(\ref{eq_915}) we learn that for any $|\xi| \leq \eps^{-2/3}$,
$$ e^{\xi^2/2} \left| \vphi_{\theta}(\xi) - e^{-\xi^2/2} \right| =
\left| e^{\xi^2/2} \prod_{j=1}^n \vphi_j \left( \theta_j \xi \right)
- 1 \right| = \left| e^{O(R_1 |\xi|^3 + \eps^4 \xi^4)} - 1 \right|
\leq C^{\prime} \left[ R_1 |\xi|^3 + \eps^4 \xi^4 \right].
 $$
 We integrate the above, to conclude that
$$ \int_{-\eps^{-2/3}}^{\eps^{-2/3}}
\left| \vphi_{\theta}(\xi) -
e^{-\xi^2/2} \right| \frac{d \xi}{|\xi|} \leq
 \int_{-\infty}^{\infty}C^{\prime} \left[ R_1 |\xi|^3 + \eps^4 \xi^4 \right]  e^{-\xi^2/2} \frac{d\xi}{|\xi|} \leq \tilde{C} \left[ R_1 + \eps^4 \right]. $$
\end{proof}

\begin{lemma} Let $\theta = (\theta_1,\ldots,\theta_n) \in S^{n-1}$. Denote, as before,
$\eps = \left( \sum_{j=1}^n \theta_j^4 \delta_j^4 \right)^{1/4}$,
and suppose that $R_2 > 0$ satisfies
\begin{equation}
 \sum_{j \in \cS}  \theta_j^2 \geq 1/8 \ \ \ \ \ \ \ \ \text{where} \ \ \ \ \ \cS =
\left \{ 1 \leq j\leq n \ ; \
  |\theta_j| \leq   R_2 / \bar{\gamma}_j^{3} \right \}.
 \label{eq_1152}
 \end{equation}
Then, whenever $\eps^{-2/3} \leq c R_2^{-1}$,
$$ \int_{\eps^{-2/3}}^{c R_2^{-1}}  \left| \vphi_{\theta}(\xi) -
e^{-\xi^2/2} \right| \frac{d \xi}{|\xi|} \leq C \eps^4. $$ The right-hand side is also an upper bound for the integral from $-c R_2^{-1}$
to $-\eps^{-2/3}$. Here $C, c
> 0$ are universal constants. \label{lem_dom2}
\end{lemma}

\begin{proof} As in the beginning of the proof of Lemma
\ref{lem_dom1}, we use Taylor's theorem. We conclude that for
$j=1,\ldots,n$,
$$ |\log \vphi_j(s) + \frac{1}{2} s^2| \leq C \bar{\gamma}_j^3 |s|^3 \ \ \ \ \
\text{when} \ \ |s| \leq 1 / \bar{\gamma}_j. $$
 Hence, for $j=1,\ldots,n$ and $s \in \RR$,
\begin{equation}
 |\vphi_j(s)| \leq \exp \left( - s^2 / 4 \right) \ \ \ \ \ \ \text{when} \ \ |s| \leq c / \bar{\gamma}_j^{3}. \label{eq_902}
\end{equation}
Let $\xi \in \RR$ be such that $|\xi| \leq c R_2^{-1}$ where $c$ is
the constant from (\ref{eq_902}). For any $j \in \cS$, we have
$|\theta_j \xi| \leq c / \bar{\gamma}_j^{3}$. Therefore,
$$ |\vphi_{\theta}(\xi)| = \prod_{j=1}^n |\vphi_j(\theta_j \xi)| \leq \prod_{j \in \cS} |\vphi_j(\theta_j \xi)|
\leq \exp \left( -\sum_{j \in \cS} \theta_j^2 \xi^2 / 4 \right) \leq
\exp \left(-\tilde{c} \xi^2 \right)
$$
where the last inequality follows from (\ref{eq_1152}). Consequently,
$$ \int_{\eps^{-2/3}}^{c
R_2^{-1}} \left| \vphi_\theta( \xi) - e^{-\xi^2/2} \right| \frac{d
\xi}{|\xi|} \leq \int_{\eps^{-2/3}}^{c R_2^{-1}} \left[
e^{-\tilde{c} \xi^2} + e^{-\xi^2 /2} \right] \frac{d \xi}{|\xi|}
\leq C \eps^{2/3} e^{-c / \eps^{4/3}} \leq \bar{C} \eps^4. $$
\end{proof}

\section{Properties of a random direction}
\label{sec3}

We retain the notation of the previous section, and our first goal
is to estimate $R_2$ from Lemma \ref{lem_dom2}. The following lemma
serves that purpose. For a random variable $Y$ and $a \in \RR$ we
write $1_{\{Y > a\}}$ for the random variable that equals one when
$Y > a$ and vanishes otherwise.

\begin{lemma} Let $M \geq 1$ and suppose that $Y$ is a non-negative random variable
with $\EE Y = 1$ and $\EE Y^2 \leq M$. Then,
\begin{enumerate}
\item[(i)] $\displaystyle \PP \left( Y \geq 1/2 \right) \geq 1 / (4M),$
\item[(ii)] $ \displaystyle \EE Y 1_{\{ Y \leq 5 M \}} \geq 4/5.$
\end{enumerate} \label{lem_859}
\end{lemma}

\begin{proof} The first inequality is due to Paley and Zygmund
(see, e.g., Kahane \cite[Section 1.6]{Kahane}). To prove (ii),
observe that
$$  \EE 1_{ \{ Y \leq 5 M \}} Y =
1- \EE 1_{ \{ Y > 5 M \}} Y \geq 1 - \frac{1}{5 M} \EE Y^2 \geq 4/5.
$$ \end{proof}

The rest of this section and the next section are devoted to the proof of Theorem \ref{thm_1517}.
Readers interested only in the proof of Theorem \ref{thm_1328} may proceed to Section \ref{sec5}.
 Suppose that $\Theta =
(\Theta_1,\ldots,\Theta_n)$ is a random vector, distributed
uniformly on the unit sphere $S^{n-1}$.

\begin{lemma} Let $\cJ \subseteq \{1,\ldots,n\}$ be a subset, denote its cardinality by $k =
\#(\cJ)$, and assume that $k \geq 4 n / 5$. Then with probability
greater than $1 - C \exp(- cn)$ of selecting the random vector
$\Theta \in S^{n-1}$,
$$ \sum_{j \in \cS}  \Theta_j^2 \geq 1/8 \ \ \ \ \ \ \ \ \text{where} \
\ \ \ \ \cS = \{ j \in \cJ;
  |\Theta_j| \leq  40 / \sqrt{n} \}. $$
Here, $C, c > 0$ are universal constants. \label{lem_prob}
\end{lemma}

\begin{proof} Let us introduce independent, standard gaussian
random variables $ \Gamma_j \ \ (j \in \cJ)$, that are independent
of the $\Theta_j$'s. Let $Z$ be a chi-square random variable with $k
= \#(\cJ)$ degrees of freedom, independent of the $\Gamma_j$'s and
the $\Theta_j$'s. Then $Z$ has the same distribution as $\sum_{j \in
\cJ} \Gamma_j^2$. Bernstein's inequality (see, e.g., Ibragimov and
Linnik \cite[Chapter 7]{IL}) yields
\begin{equation}
 \PP \left( \frac{k}{2} \leq Z \leq 2 k \right) \geq 1 - C \exp(-c k) \geq 1 - \tilde{C} \exp(-\tilde{c} n).
\label{eq_1128} \end{equation}
 Observe that  the random variables $(\Gamma_j)_{j \in \cJ}$ have
exactly the same joint distribution as the random variables
$(\sqrt{Z} \Theta_j)_{j \in \cJ}$. Therefore, in order to prove the
lemma, it suffices to show that with probability greater than $1 - C
\exp(- cn)$,
$$ \sum_{j \in \cS}  \Gamma_j^2 \geq n/2 \ \ \ \ \ \ \ \ \text{where} \
\ \ \ \ \cS = \{ j \in \cJ;
  |\Gamma_j| \leq  20 \}. $$
Denote $Y_j= \Gamma_j^2 1_{\{|\Gamma_j| \leq 20\}}$. Then $(Y_j)_{j
\in \cJ}$ are independent, identically-distributed random variables,
and our goal is to prove that
\begin{equation}
 \PP \left( \sum_{j \in \cJ} Y_j \geq n/2 \right) \geq 1 - C \exp(-
cn). \label{eq_923}
\end{equation}
Since $\EE \Gamma_j^4 = 3$, then Lemma \ref{lem_859}(ii) yields that
$$ \EE Y_j \geq 4/5, \ \ \ \ \ \ \ \ \text{and clearly} \ \ \ \ \ \ \ Var(Y_j)
\leq \EE Y_j^2 \leq \EE \Gamma_j^2 = 3 $$  for $j \in \cJ$.
According to Bernstein's inequality, \begin{equation} \PP \left(
\sum_{j \in \cJ} Y_j \leq \frac{4 k}{5} - t \sqrt{3 k} \right) \leq
C \exp(-c t^2) \ \ \ \ \ \ \ \ \text{for any} \ t \geq 0.
\label{eq_1305} \end{equation} Recall that $k / n \geq 4/5$.
Inequality (\ref{eq_923}) follows by setting $t = \sqrt{n / 200}$ in
(\ref{eq_1305}). \end{proof}

\begin{corollary} Set $R = 200 \delta^2 / \sqrt{n}$.
Then with probability greater than $1 - C \exp(- cn)$ of selecting
the random vector $\Theta \in S^{n-1}$,
$$ \sum_{j \in \cS}  \Theta_j^2 \geq 1/8 \ \ \ \ \ \ \ \ \text{where} \
\ \ \ \ \cS = \{ 1 \leq j \leq n;
  |\Theta_j| \leq  R / \bar{\gamma}_j^{3} \}. $$
Here, $C, c > 0$ are universal constants. \label{cor_prob}
\end{corollary}

\begin{proof} Denote $\cJ = \{ 1 \leq j \leq n ;
\bar{\gamma}_j^3 \leq 5 \delta^{2} \}$. Then,
$$ \delta^4 = \frac{1}{n} \sum_{j=1}^n \delta_j^4 \geq
\frac{1}{n} \sum_{j=1}^n \bar{\gamma}_j^6 \geq \frac{1}{n} \sum_{j
\not \in \cJ} \bar{\gamma}_j^6 > \frac{n - \#(\cJ)}{n} (5
\delta^2)^2.
$$ Denoting $k =
\#(\cJ)$, we thus see that $k / n \geq 24/25 \geq 4/5$. For $j \in
\cJ$, we have $40 / \sqrt{n} \leq R / \bar{\gamma}_j^3$. In order
to prove the lemma, it therefore suffices to show that with
probability greater than $1 - C \exp(- cn)$,
$$ \sum_{j \in \cS}  \Theta_j^2 \geq 1/8 \ \ \ \ \ \ \ \ \text{where} \
\ \ \ \ \cS = \{ j \in \cJ;
  |\Theta_j| \leq  40 / \sqrt{n} \}. $$
This is precisely the content of Lemma \ref{lem_prob}.
\end{proof}

\medskip Our  goal is to bound the integral in
(\ref{eq_1617}). Lemma \ref{lem_dom1} and Lemma \ref{lem_dom2} (with
the help of Corollary \ref{cor_prob}) control the contribution of
the interval $[-c \sqrt{n} / \delta^2, c \sqrt{n} / \delta^2]$. Next
we aim at bounding the contributions of $\xi \in \RR$ with $ c
\sqrt{n} / \delta^2 \leq |\xi| \leq n / \delta^4$. Denote
$$ J_n(\xi) = \EE \exp(-i \xi \Theta_1) \ \ \ \ \ \ \ \ \ \ \ (\xi \in \RR). $$
The function $J_n$ is even and real-valued, and is related to the
Bessel function of order $n/2-1 $.

\begin{lemma} We have
\begin{equation}
 J_n(\xi) \leq 1 - c \min \left \{ \xi^2  /n, 1 \right \} \ \ \ \ \ \ \ \ \ \text{for all} \ \xi \in \RR,
 \label{eq_1555}
 \end{equation}
where $c > 0$ is a universal constant. \label{lem_1647}
\end{lemma}

\begin{proof} Since $\EE (\sqrt{n} \Theta_1)^2 = 1$ and $\EE
(\sqrt{n} \Theta_1)^4 \leq C$, then Taylor's theorem yields
 $$ J_n(\sqrt{n} \tau) = 1 - \frac{\tau^2}{2} + O(\tau^4) $$
 for $|\tau| \leq 1$, as the odd moments vanish. This implies (\ref{eq_1555}) for $|\xi| \leq c \sqrt{n}$. The
density $f_n$ of the random variable $\Theta_1$ vanishes outside
$[-1,1]$, and is proportional to $t \mapsto (1 - t^2)^{(n-3)/2}$ on
$[-1,1]$. Denote $g_n(t) = n^{-1/2} f_n(n^{-1/2} t)$. Then
$$ \int_{-\infty}^{\infty} \left| g_n(t) - \frac{1}{\sqrt{2 \pi}} \exp \left(-t^2 / 2 \right)
\right| dt \stackrel{n \rightarrow \infty}{\longrightarrow} 0, $$ as
be may verified routinely (see, e.g., Diaconis and Freedman
\cite{DF2} for quantitative bounds).
 Therefore the Fourier transform satisfies
$$ \sup_{\xi \in \RR} \left| J_n(\sqrt{n} \xi) - \exp(-\xi^2/2) \right| \stackrel{n \rightarrow \infty}{\longrightarrow} 0
$$
 which implies (\ref{eq_1555}) in the range $|\xi| \geq c \sqrt{n}$.
\end{proof}

\begin{lemma}
Let $j=1,\ldots,n$. Then, for any $\tau \in \RR$,
$$ \EE |\vphi_j(\tau \Theta_j)|^2 \leq 1 - c \min \left \{ \tau^2 / n, \delta_j^{-4}  \right \}, $$
where $c > 0$ is a universal constant.
\label{lem_1627}
\end{lemma}

\begin{proof} As before, denote by $f_n$ the density of the
random variable $\Theta_1$. Then,
$$
\EE |\vphi_j(\tau \Theta_j)|^2 = \int_{-\infty}^{\infty} |\vphi_j(\tau \xi)|^2 f_n(\xi) d \xi. $$
Let $\tilde{X_j}$ be an independent copy of $X_j$.
Define $Y = X_j - \tilde{X_j}$, a symmetric random variable. Then
the Fourier transform of $Y$ is
$$ \EE \exp(-i \xi Y) = \EE \exp(-i \xi X_j) \overline{\EE \exp(-i \xi \tilde{X_j})} = \vphi_j(\xi)
\overline{\vphi_j(\xi)} = |\vphi_j(\xi)|^2. $$ Hence the function $
|\vphi_j(\tau \xi)|^2$ is the Fourier transform of the random
variable $\tau Y$. Recall that $J_n(\xi)$ is the Fourier transform
of the density $f_n$. The central observation is that according to
the Plancherel theorem,
 $$ \EE |\vphi_j(\tau \Theta_j)|^2 = \int_{-\infty}^{\infty} |\vphi_j(\tau \xi)|^2 f_n(\xi) d \xi
 = \EE J_n(\tau Y) \leq 1 -
 c \EE \min \{ \tau^2 Y^2 / n, 1 \}, $$
 where the last inequality is the content of Lemma \ref{lem_1647}.
Denote $r = \tau^2 / n$ and $Z = Y^2 / 2$. In order to complete
 the proof of the lemma, it suffices to show that for any $r \geq 0$,
\begin{equation}
\EE \min \{ r Z, 1 \} \geq c  \min \left \{ r, \delta_j^{-4} \right \}.
\label{eq_1717}
\end{equation}
The left-hand side of (\ref{eq_1717}) is non-decreasing in $r$,
hence it is enough to prove (\ref{eq_1717}) when $r \leq
\delta_j^{-4} / 10$. Since $Y = X_j - \tilde{X_j}$ and $Z = Y^2/2$,
then $ \EE Z = 1$ and $\EE Z^2 = (3 + \delta_j^4) / 2 \leq 2 \delta_j^4$.
According to Lemma \ref{lem_859}(ii), $\EE 1_{ \{ Z \leq 10
\delta_j^4 \}} Z \geq 4/5$.  Therefore, for $0 \leq r \leq
\delta_j^{-4} / 10$,
$$ \EE \min \{ r Z, 1 \} \geq \EE 1_{ \{ Z \leq 10 \delta_j^4 \}} \min \{ r Z, 1 \}
= r \EE 1_{ \{ Z \leq 10 \delta_j^4 \}} Z \geq r / 2, $$ and
(\ref{eq_1717}) follows.
 The lemma is thus proven. \end{proof}

\medskip When the dimension $n$ is large,
the random variables $\Theta_1,\ldots,\Theta_n$ are ``approximately
independent''. One would thus expect that usually, for
functions $f_1,\ldots,f_n : \RR \rightarrow \CC$,
\begin{equation}
 \EE \prod_{j=1}^n f_j (\Theta_j) \approx  \prod_{j=1}^n \EE f_j
 (\Theta_j).
 \label{eq_1151}
 \end{equation}
The most straightforward way to obtain estimates in the spirit of
(\ref{eq_1151})
 is to compare the distribution
of $\Theta$ with that of a gaussian random vector of the same
expectation and covariance as in the proof of Lemma \ref{lem_prob}
above. Even though this approach works well in our present context,
we prefer to invoke below a recent inequality due to  Carlen, Lieb
and Loss \cite{CLL}.
 This inequality provides a particularly elegant way to exploit the ``approximate
independence'' of $\Theta_1,\ldots,\Theta_n$. It states
that for any non-negative, measurable functions $f_1,\ldots,f_n: [-1,1]
\rightarrow \RR$,
\begin{equation}
 \EE \prod_{j=1}^n f_j(\Theta_j)
\leq \prod_{j=1}^n \left( \EE f_j(\Theta_j)^2 \right)^{1/2}.
\label{eq_1401}
\end{equation}
See Barthe, Cordero-Erausquin, Ledoux and Maurey \cite{BCLM} for ramifications of the
Brascamp-Lieb type inequality (\ref{eq_1401}). Recall that $\vphi_1,\ldots,\vphi_n$
are the Fourier transforms of the independent random variables $X_1,\ldots,X_n$.

\begin{lemma} Let $\alpha > 0$ and assume that $\alpha \sqrt{n} / \delta^2 \leq n / \delta^4$.
Then, with probability greater than $1 - C(\alpha)
\exp(-c(\alpha) n / \delta^4)$ of selecting
 $(\Theta_1,\ldots,\Theta_n) \in S^{n-1}$,
$$ \int_{\alpha \sqrt{n} / \delta^2}^{n / \delta^4} \left| \prod_{j=1}^n \vphi_j( \Theta_j \xi) -
e^{-\xi^2/2} \right| \frac{d \xi}{|\xi|} \leq C(\alpha)
\exp(-c(\alpha) n / \delta^4) \leq \frac{\bar{C}(\alpha)
\delta^4}{n}.
$$ The right-hand side is also an upper bound for the integral from $-n/\delta^4$ to $-\alpha \sqrt{n} / \delta^2$.
Here $C(\alpha), \bar{C}(\alpha), c(\alpha)
> 0$ are constants depending solely on $\alpha$. \label{lem_dom3}
\end{lemma}

\begin{proof} Lemma \ref{lem_1627} and (\ref{eq_1401}) imply that
for any $\xi \in \RR$,
$$
 \EE \left| \prod_{j=1}^n \vphi_j(\Theta_j \xi) \right| \leq
\prod_{j=1}^n \sqrt{ \EE |\vphi_j(\Theta_j \xi)|^2} \leq
\prod_{j=1}^n \left( 1 - c \min \left \{ \xi^2 / n, \delta_j^{-4}
\right \} \right). $$ Denote $\cJ = \{ 1 \leq j \leq n ; \delta_j
\leq 2 \delta \}$. Repeating a simple argument, we have
$$ \delta^4 = \frac{1}{n} \sum_{j=1}^n \delta_j^4 \geq
\frac{1}{n} \sum_{j \not \in \cJ} \delta_j^4 \geq \frac{n -
\#(\cJ)}{n} 16 \delta^4, $$ hence $\#(\cJ) \geq n / 2$. For any $\xi
\in \RR$,
\begin{eqnarray*}
\lefteqn{ \EE \left| \prod_{j=1}^n \vphi_j(\Theta_j \xi) \right|
\leq  \prod_{j=1}^n \left( 1 - \tilde{c} \min \left \{ \xi^2 / n,
\delta_j^{-4} \right \} \right) \leq \prod_{j \in \cJ} \left( 1 -
\tilde{c} \min \left \{ \xi^2 / n, \delta_j^{-4} \right \} \right) }
\\ & \leq & \left( 1 - c \min \left \{ \xi^2 / n, \delta^{-4} \right \} \right)^{n/2} \leq
 \exp( -c^{\prime} \min \left \{ \xi^2, n  / \delta^4 \right \} ).
 \phantom{aaaaaaaaaaaaaaaaaaaa}
\end{eqnarray*}
Therefore,
\begin{eqnarray*}
\lefteqn{ \EE \int_{\alpha \sqrt{n} / \delta^2 }^{ n / \delta^4} \left| \prod_{j=1}^n \vphi_j(\Theta_j \xi) -
e^{-\xi^2/2} \right| \frac{d \xi}{|\xi|} } \\ & \leq & \frac{
\delta^2}{\alpha \sqrt{n}}
 \int_{ \alpha \sqrt{n} / \delta^2}^{n / \delta^4} e^{ -\tilde{c} \min \left \{ \xi^2, n / \delta^4 \right \}  } + e^{-\xi^2/2}
 d \xi  \leq
C(\alpha) e^{-c(\alpha) n / \delta^4}.
 \end{eqnarray*}
From the Chebyshev inequality,
$$ \PP \left(  \int_{\alpha \sqrt{n} / \delta^2 }^{ n / \delta^4}  \left| \prod_{j=1}^n \vphi_j(\Theta_j \xi) -
e^{-\xi^2/2} \right| \frac{d \xi}{|\xi|} \geq \sqrt{ C(\alpha)
e^{-c(\alpha) n / \delta^4}} \right) \leq \sqrt{
C(\alpha) e^{-c(\alpha) n / \delta^4}}.
$$
\end{proof}

\medskip The results obtained so far may be summarized as
follows:

\begin{lemma} There exists a subset $\cF \subseteq S^{n-1}$
with $\sigma_{n-1}(\cF) \geq 1 - C \exp(-c n / \delta^4)$ such that
for any $\theta = (\theta_1,\ldots,\theta_n) \in \cF$ with $\sum_{j=1}^n \theta_j^4 \delta_j^4 \leq 1$,
\begin{equation}
 \int_{-n / \delta^4}^{n / \delta^4} \left| \prod_{j=1}^n \vphi_j(
\theta_j \xi) - e^{-\xi^2/2} \right| \frac{d \xi}{|\xi|} \leq C
\left( \sum_{j=1}^n \theta_j^4 \delta_j^4 + \left| \sum_{j=1}^n
\gamma_j^3 \theta_j^3 \right| + \frac{\delta^4}{n}  \right),
\label{eq_1023}
\end{equation}
 where
$C > 0$ is a universal constant. \label{lem_1028}
\end{lemma}

\begin{proof} Let $\cF_1 \subseteq S^{n-1}$ be the set of directions
with $\sigma_{n-1}(\cF_1) \geq 1 - C \exp(-c n)$ whose existence is
guaranteed by Corollary \ref{cor_prob}. Assume that $\theta \in
\cF_1$ is such that $\sum_j \theta_j^4 \delta_j^4 \leq 1$. The bound
(\ref{eq_1023}) is the culmination of three arguments: Lemma
\ref{lem_dom1} controls the contribution of $|\xi| \leq
\eps^{-2/3}$, for $\eps = \left( \sum_{j=1}^n \theta_j^4 \delta_j^4
\right)^{1/4} \leq 1$. Thanks to the definition of $\cF_1$, Lemma
\ref{lem_dom2} with $R_2 = 200 \delta^2 / \sqrt{n}$ provides an
upper bound for the contribution up to $|\xi| \leq c \sqrt{n} /
\delta^2$. We conclude with an application of Lemma \ref{lem_dom3},
with $\alpha$ being a universal constant. Denote by $\cF_2 \subseteq
S^{n-1}$ the set with $\sigma_{n-1}(\cF_2) \geq 1 - C \exp(-c n /
\delta^4)$ whose existence is guaranteed by Lemma \ref{lem_dom3}.
Setting $\cF = \cF_1 \cap \cF_2$, we see that (\ref{eq_1023}) holds
for any $\theta \in \cF$ with $\sum_j \theta_j^4 \delta_j^4 \leq 1$.
\end{proof}

\begin{corollary}
There exists a subset
$\cF_1 \subseteq S^{n-1}$ with $\sigma_{n-1}(\cF_1) \geq 1 - C
\exp(-c n / \delta^4)$ with the following property: For any $\theta
= (\theta_1,\ldots,\theta_n) \in \cF_1$ and $t \in \RR$,
\begin{equation}
  \left| \PP \left( \sum_{j=1}^n \theta_j X_j \leq
t \right) \, - \, \Phi(t) \right| \leq C \left[ \sum_{j=1}^n
\delta_j^4 \theta_j^4  + \left| \sum_{j=1}^n \gamma_j^3 \theta_j^3
\right| + \frac{\delta^4}{n}  \right]. \label{eq_2133}
\end{equation}
Here $\Phi(t) = (2 \pi)^{-1/2} \int_{-\infty}^{t} e^{-t^2/2}dt$ and
$C, c
> 0$ are universal constants. \label{cor_1123}
\end{corollary}

\begin{proof} It is enough to consider $\theta$ for which $\sum_j
\delta_j^4 \theta_j^4 \leq 1$, as otherwise (\ref{eq_2133}) holds
trivially. The bound (\ref{eq_2133}) is thus an immediate
consequence of the smoothing inequality (\ref{eq_958}) and Lemma
\ref{lem_1028}. \end{proof}

\section{Deviation inequalities}
\label{sec4}

It remains to deduce Theorem \ref{thm_1517} from Corollary
\ref{cor_1123}. To that end, we need to analyze the terms $
\sum_{j=1}^n \gamma_j^3 \theta_j^3$ and $\sum_{j=1}^n \delta_j^4
\theta_j^4 $ appearing in Corollary \ref{cor_1123}. We would like to
get a bound in (\ref{eq_1412}) of the form $C(\rho) \delta^4 / n$,
where $C(\rho)$ depends on $\rho$ solely. This is the reason we use
the following crude lemma.

\begin{lemma} Suppose that $(\Theta_1,\ldots,\Theta_n) \in S^{n-1}$ is a random vector,
distributed uniformly on $S^{n-1}$. Then, for any $t \geq 0$,
\begin{equation}
\PP \left( \left| \sum_{j=1}^n \gamma_j^3 \Theta_j^3 \right| \geq t
\frac{\delta^4}{n} \right) \leq C \exp \left( -c t^{2/3} \right),  \label{eq_1719}
 \end{equation}
and additionally,
\begin{equation}
 \PP \left( \sum_{j=1}^n \delta_j^4  \Theta_j^4 \geq t \frac{\delta^4}{n}
  \right) \leq C \exp \left( -c \sqrt{t}  \right).
 \label{eq_1341}
 \end{equation}
Here $C, c > 0$ are universal constants. \label{lem_927}
\end{lemma}

\begin{proof} Introduce independent, standard gaussian random
variables $ \Gamma_1,\ldots,\Gamma_n$ that are independent of the
$\Theta_j$'s. Let $Z$ be a chi-square random variable with $n$
degrees of freedom, independent of the $\Gamma_j$'s and
$\Theta_j$'s. As in (\ref{eq_1128}), we know that $n/2 \leq Z \leq 2
n$ with probability greater than $1 - C \exp(-c n)$. Thus,
\begin{equation}
  \PP \left( \left| \sum_{j=1}^n \gamma_j^3 \Theta_j^3 \right| \geq t \right) =
 \PP \left( \left| \frac{\sum_{j=1}^n \gamma_j^3 \Gamma_j^3}{Z^{3/2}} \right| \geq t \right) \leq
\PP \left( \left| \sum_{j=1}^n \gamma_j^3 \Gamma_j^3 \right| \geq \frac{n^{3/2} t}{4} \right) + C e^{-cn}.
\label{eq_1116}
\end{equation}
The random variable $Y = \sum_{j=1}^n \gamma_j^3 \Gamma_j^3$ is the
sum of independent, mean zero random variables. We will apply a
moment inequality we learned from Adamczak, Litvak, Pajor and
Tomczak-Jaegermann \cite[Section 3]{AL}, which builds upon previous
work by Hitczenko, Montgomery-Smith, and Oleszkiewicz \cite{H}.
Recall that $\EE \exp(c \Gamma_j^2) \leq 2$ for a universal constant
$c > 0$. In the terminology of \cite{AL}, the random variables
$\Gamma_1^3,\ldots,\Gamma_n^3$ are random variables of class
$\psi_{2/3}$, hence for any $p \geq 2$,
$$ \left( \EE |Y|^p \right)^{1/p} \leq C \left( p^{1/2} \sqrt{\sum_{j=1}^n \gamma_j^6}
+ p^{3/2} \left( \sum_{j=1}^n |\gamma_j|^{{3p}} \right)^{1/p} \right)
\leq \tilde{C} p^{3/2} \sqrt{\sum_{j=1}^n \gamma_j^6} \leq \bar{C} p^{3/2} \sqrt{n}  \delta^2,
$$
as $\gamma_j^6 \leq \delta_j^4$ for all $j$. According to the
Chebyshev inequality, for any $t \geq C \delta^2 \sqrt{n}$,
$$
 \PP \left( \left| \sum_{j=1}^n \gamma_j^3 \Gamma_j^3 \right| \geq t \right) \leq \frac{\EE |Y|^p}{t^p}
\leq \frac{ \left( \bar{C} p^{3/2} \sqrt{n} \delta^2 \right)^p
}{t^p} \leq e^{-p} \leq  \exp \left(-\tilde{c}
\frac{t^{2/3}}{(\delta^4 n)^{1/3}} \right) $$ where $p = c t^{2/3} /
(\delta^4 n)^{1/3}$ for an appropriate small universal constant $c
> 0$. From (\ref{eq_1116}) and the last inequality,
\begin{equation}
 \PP \left( \left| \sum_{j=1}^n \gamma_j^3 \Theta_j^3 \right| \geq t \frac{\delta^2}{n} \right)
 \leq C \exp \left( -c t^{2/3}
\right) + C \exp(-cn) \ \ \ \ \ \ \ \text{for all} \ t > C. \label{eq_1127}
\end{equation}
According to the Cauchy-Schwartz inequality in (\ref{eq_2210}) we
always have $\left| \sum_{j=1}^n \gamma_j^3 \Theta_j^3 \right| \leq
\sqrt{ \sum_{j=1}^n \delta_j^4 \Theta_j^4} \leq \sqrt{n} \delta^2$
 and hence the probability on the left-hand side of (\ref{eq_1127})
vanishes for $t \geq n^{3/2}$. We may thus deduce (\ref{eq_1719}) from
(\ref{eq_1127}). Inequality (\ref{eq_1341}) is proven in a similar
vein: Denote $W = \sum_{j=1}^n \delta_j^4 \left[ \Gamma_j^4 -
3 \right]$. Then, $\EE W = 0$ and for any $t \geq 0$,
\begin{equation}
 \PP \left( \sum_{j=1}^n \delta_j^4 \Theta_j^4 \geq 12 \sum_{j=1}^n \frac{\delta_j^4}{n^2}
+ t \right) \leq \PP(W \geq n^2 t / 4) + C \exp(- cn).
\label{eq_1013}
\end{equation}
The random variables $\Gamma_1^4 - 3,\ldots,\Gamma_n^4 - 3$ are
independent random variables of class $\psi_{1/2}$. Again, using the
inequality from \cite[Section 3]{AL} we see that for $p \geq 2$,
$$ \left( \EE |W|^p \right)^{1/p} \leq C \left( p^{1/2} \sqrt{\sum_{j=1}^n \delta_j^8}
+ p^{2} \left( \sum_{j=1}^n \delta_j^{{4p}} \right)^{1/p} \right)
\leq \tilde{C} p^{2} \sum_{j=1}^n \delta_j^4 = \tilde{C} p^{2} n \delta^4.
$$
Using the Chebyshev inequality, as before, we deduce that for any $t
> C$,
\begin{equation}
 \PP(W \geq t n \delta^4) \leq C \exp \left(-c \sqrt{t} \right).
\label{eq_912}
\end{equation}
Inequalities (\ref{eq_1013}) and (\ref{eq_912}) lead to the bound
$$ \PP \left( \sum_{j=1}^n \delta_j^4 \Theta_j^4 \geq t \frac{\delta^4}{n}
\right) \leq C \exp \left(-c \sqrt{t} \right) + C \exp(- cn) \ \ \ \
\ \ \text{for all} \ t \geq 15. $$ Since  with probability one
$\sum_{j=1}^n \delta_j^4 \Theta_j^4 \leq n \delta^4$, the bound
(\ref{eq_1341}) follows. \end{proof}

\medskip \emph{Proof of Theorem \ref{thm_1517}.} We may assume that
$(\log 1/\rho)^2 \leq \tilde{c} n / \delta^4$, for a small universal
constant $\tilde{c} > 0$, since otherwise the conclusion
(\ref{eq_1412}) of the theorem is trivial, for an appropriate choice
of a universal constant $C$ in Theorem \ref{thm_1517}. Therefore,
$$ \rho \geq \exp \left(-\sqrt{\tilde{c} n / \delta^4} \right) \geq \exp
\left(-c n / \delta^4 \right) $$
where $c > 0$ is the constant from Corollary \ref{cor_1123}.
Let $\cF_1 \subseteq S^{n-1}$ be the subset of the sphere with $\sigma_{n-1}(\cF_1)
\geq 1 - C \exp(-c n / \delta^4) \geq 1 - C \rho$ whose existence
is guaranteed by Corollary \ref{cor_1123}. According to Lemma \ref{lem_927}, there exists
a subset $\cF_2 \subseteq S^{n-1}$ with $\sigma_{n-1}(\cF_2) \geq 1 - \rho$
such that for any $\theta = (\theta_1,\ldots,\theta_n) \in \cF_2$,
\begin{equation}
  \left| \sum_{j=1}^n \gamma_j^3 \theta_j^3 \right| + \sum_{j=1}^n \delta_j^4  \theta_j^4 \leq
\tilde{C} \left( \log \frac{1}{\rho} \right)^{3/2} \frac{\delta^4}{n} +
    \tilde{C} \left( \log \frac{1}{\rho} \right)^2 \frac{\delta^4}{n}
\leq \bar{C} \left( \log \frac{1}{\rho} \right)^2 \frac{\delta^4}{n}. \label{eq_934}
\end{equation}
Denote $\cF = \cF_1 \cap \cF_2$. Then $\sigma_{n-1}(\cF) \geq 1 -
C^{\prime} \rho$. Furthermore, according to Corollary \ref{cor_1123}
and to (\ref{eq_934}), the desired bound (\ref{eq_1412}) holds for
any $\theta \in \cF$, with $C(\rho) \leq \hat{C} (\log 1 / \rho)^2$.
\hfill $\square$

\medskip \emph{Remark.} There is some wiggle room in the bound for
$C(\rho)$ in Theorem \ref{thm_1517}. One easily notices  that the
bounds stated in Lemma \ref{lem_927} are, in many cases, quite weak:
When all the $\delta_j$ are comparable, a better analysis of the
moment inequality from \cite[Section 3]{AL} leads to a sub-gaussian
tail, at least in some range. If one is interested in a version of
Theorem \ref{thm_1517} where $\delta^4 = \sum_j \EE X_j^4 / n$ is
replaced by the larger quantity $\max_j \EE X_j^4$, finer analogs of
Lemma \ref{lem_927} may be employed. For such a version of Theorem
\ref{thm_1517}, the power of the logarithm in the bound for
$C(\rho)$ may essentially be improved, from $2$ to $1/2$, at least
for $\rho$ in some range.

\section{Explicit, universal coefficients}
\label{sec5}

This section is devoted to the proof of Theorem \ref{thm_1328} and
related statements. We assume that the independent random variables
$X_1,\ldots,X_n$ are identically distributed, and that they have the same
 distribution as a certain random variable $X$. This random variable $X$ has mean zero,
variance one, and we denote $\delta = (\EE X^4)^{1/4}$. Its Fourier
transform is
$$ \vphi(\xi) = \EE \exp(-i \xi X) \quad \quad \quad \quad (\xi \in \RR). $$
As before we fix $\theta \in S^{n-1}$ and set
\[ \vphi_\theta(\xi) = \prod_{j=1}^n \vphi(\theta_j \xi) \quad \quad \quad \quad (\xi \in \RR). \]
Let $\tilde{X}$ be an independent copy of $X$, and define $Y = X -
\widetilde{X}$. The next three lemmas bound an integral of
$|\vphi_\theta|$ in terms of a certain arithmetic property of
$\theta$. The property is quite similar to the one introduced by
Rudelson and Vershynin \cite{RV}; we closely follow their
presentation, taking into account several simplifications proposed
in \cite{FS}.

\begin{lemma} For any $\xi \in \RR$,
$$ |\vphi_\theta(\xi)|  \leq  \exp \left\{ -4 \EE d^2 \left( \frac{\xi Y}{2 \pi} \theta, \ZZ^n\right)
 \right\}. $$ \label{lem_1101}
\end{lemma}

\begin{proof} It is easily verified that
\[ \cos \theta \leq 1 - \frac{2}{\pi^2} \, d^2(\theta, 2 \pi \ZZ) = 1 - 8 d^2 \left(\theta / (2 \pi), \ZZ \right)~ \quad \quad (\theta \in \RR).\]
As in the proof of Lemma \ref{lem_1627}, for any $\xi \in \RR$,
\begin{eqnarray*}
\lefteqn{|\vphi(\xi)|^2 = \EE \exp(- i \xi Y) = \EE \cos (\xi Y)  } \\ & \leq & 1 - 8 \EE d^2 \left(\xi Y / (2 \pi), \ZZ \right) \leq \exp \left \{ -8 \EE d^2 \left(\xi Y / (2 \pi), \ZZ \right) \right \}.
\end{eqnarray*}
Therefore,
$$ |\vphi_\theta(\xi)| = \prod_{j=1}^n |\vphi(\xi \theta_j)| \leq \exp \left \{ -4 \EE \sum_{j=1}^n
d^2 \left(\xi \theta_j Y / (2 \pi), \ZZ \right) \right \}. $$
\end{proof}

The following lemma summarizes a few properties of the even function
$$ S(\xi) = \sqrt{\EE d^2 \left( \frac{\xi Y}{2 \pi} \theta, \ZZ^n\right)} \quad \quad \quad \quad (\xi \in \RR). $$
Recall the definition  of $\cN(\theta)$, for a unit vector $\theta \in S^{n-1}$.

\begin{lemma} For any $\xi_1, \xi_2 \in \RR$,
\begin{equation}
 S(\xi_1 + \xi_2) \leq S(\xi_1) + S(\xi_2). \label{eq_912}
 \end{equation}
Furthermore, denote $R = \cN(\theta) \geq 1$.
Then, for any $|\xi| \leq n / (R \delta^4)$,
\begin{equation}
 S(\xi) \geq c \min \left \{ |\xi|, \frac{n / (R \delta^4)}{|\xi|} \right \},
\label{eq_26}
\end{equation}
where $c > 0$ is a universal constant.
\label{lem_1055}
\end{lemma}

\begin{proof} Note that for any $x, y \in \RR^n$,
$$ d(x + y, \ZZ^n) \leq d(x, \ZZ^n) + d(y, \ZZ^n). $$
The inequality (\ref{eq_912}) thus follows from the Cauchy-Schwartz inequality. Let us move to the proof of (\ref{eq_26}). From the definition of $\cN(\theta)$,
\begin{equation}
 d(\xi \theta, \ZZ^n) \geq \frac{1}{10} \min \left \{ |\xi|, \frac{n / R}{ |\xi|} \right \} = \frac{|\xi|}{10} \min \left \{ 1, \frac{n/R}{\xi^2} \right \} \ \ \ \ \ \ \text{for all} \ |\xi| \leq n. \label{eq_1037}
\end{equation}
Since $Y = X - \tilde{X}$ then $\EE Y^2 = 2$ and $\EE Y^4 =
2\delta^4 + 6 \leq 8 \delta^4$. Denote $\tilde{Y} = Y 1_{|Y| \leq 5 \delta^2}$. According to Lemma \ref{lem_859}(ii), $$ \EE \tilde{Y}^2 \geq 4/5. $$
For any $|\xi| \leq n / \delta^4$ we have $|\tilde{Y} \xi| / (2 \pi) \leq \delta^2 \xi \leq n$ and (\ref{eq_1037}) yields
\begin{eqnarray*}
 \lefteqn{ \EE d^2 \left( \frac{\xi Y}{2 \pi} \theta, \ZZ^n\right) \geq \EE d^2 \left( \frac{\xi \tilde{Y}}{2 \pi} \theta, \ZZ^n\right)   \geq
\frac{1}{40 \pi^2} \EE \left( \xi \tilde{Y} \min \left \{ 1, \frac{n / R}{|\tilde{Y} \xi|^2 / (4 \pi^2)} \right \} \right)^2 } \\ & \geq &
 \frac{1}{40 \pi^2}  \xi^2 \min \left \{1, \frac{n^2 / R^2}{ \delta^8 \xi^4} \right \} \EE \tilde{Y}^2 \geq \frac{1}{800}
\min \left \{\xi^2, \frac{n^2 / R^2}{ \delta^8 \xi^2} \right \}. \phantom{paul is dead}
\end{eqnarray*}
Therefore (\ref{eq_26}) holds  for all $|\xi| \leq n / \delta^4$.
\end{proof}

\begin{lemma} Let $0 < \alpha < 1, T \geq 1$ and suppose that $f: \RR \rightarrow [0, \infty)$ is an even, measurable function
which satisfies
\begin{equation}
 f(\xi_1 + \xi_2) \leq f(\xi_1) + f(\xi_2) \quad \quad (\xi_1, \xi_2 \in \RR) \label{eq_1153}
\end{equation}
and
\begin{equation}
 f(\xi) \geq \alpha \min \left \{ |\xi|, \frac{T}{|\xi|} \right \}
\quad \quad \text{for any} \ |\xi| \leq T.
\label{eq_1516}
\end{equation}
Then,
$$ \int_{T^{1/6} \leq |\xi| \leq T} \exp \left \{ -f^2(\xi) \right \} \frac{d\xi}{|\xi|} \leq \frac{C }{\alpha^{6} T} $$
where $C > 0$ is a universal constant.
\label{lem_1103}
\end{lemma}

\begin{proof}
Fix $r > 0$. Denote
$$ A_{r} = \{ T^{1/2} \leq \xi \leq T ; r \leq f(\xi) < 2 r \}. $$ Then  $A_{r}
\subset [\alpha T / (2 r), T]$, thanks to (\ref{eq_1516}). Furthermore, let $\xi_1, \xi_2 \in
A_{r}$. We learn from (\ref{eq_1153}) and from the fact that $f$ is even that
$f(\xi_1 - \xi_2) \leq f(\xi_1) + f(\xi_2) \leq 4 r$. According to (\ref{eq_1516}),
either $|\xi_1 - \xi_2| \leq 4 r / \alpha$, or else
$$ |\xi_1 - \xi_2| \geq \alpha T  / f (\xi_1 - \xi_2) \geq \alpha T / (4 r).
$$
Therefore, the  set $A_{r}$ can be covered by closed intervals of length at most
$4 r / \alpha$, and the distance between two such intervals is at least
$\alpha T / (4 r)$. For this specific purpose, the distance between two closed intervals
means the distance between their left-most points.
Since $A_r \subset [\alpha T / (2 r), T]$ then the number of such intervals
is at most $4 r / \alpha + 1$.  Consequently, for any $r > 0$,
\begin{equation}
 \int_{A_{r}} \exp \left ( - f^2(\xi) \right )
\frac{d\xi}{\xi} \leq \left( \frac{4 r}{\alpha} +1 \right) \cdot \frac{4 r }{\alpha} \cdot \exp( -r^2) \cdot \frac{2r}{\alpha T} =  \frac{8 r^2 (4 r + \alpha)}{\alpha^3 T} \exp( -r^2). \label{eq_958_}
\end{equation}
 From Fubini's theorem,
$$ \int_{\sqrt{T} \leq |\xi| \leq T} e^{ - f^2(\xi) } \frac{d\xi}{|\xi|} = 2 \sum_{i=-\infty}^{\infty} \int_{A_{2^i}} e^{ - f^2(\xi) } \frac{d\xi}{|\xi|}  \leq 4  \int_0^{\infty} \left(
\int_{A_r}  e^{ - f^2(\xi) } \frac{d\xi}{|\xi|}  \right) \frac{dr}{r}.
$$
We plug in the information from (\ref{eq_958_}) to obtain the bound
\begin{equation}
 \int_{\sqrt{T} \leq |\xi| \leq T} \exp \left \{  - f^2(\xi) \right \} \frac{d\xi}{|\xi|}  \leq C \int_0^{\infty}
\frac{r^2 (r + \alpha)}{\alpha^3 T} \exp \{ - r^2 \} \frac{dr}{r} \leq \frac{\tilde{C}}{ \alpha^{3} T}.  \label{eq_1059}
\end{equation}
Additionally, according to (\ref{eq_1516}),
\begin{equation}
 \int_{T^{1/6} \leq |\xi| \leq T^{1/2}} e^{ - f^2(\xi) } \frac{d\xi}{|\xi|}
\leq  2 \int_{T^{1/6}}^{\infty} e^{-\alpha^2 \xi^2 } \frac{d \xi}{|\xi|} \leq  \frac{C}{\alpha T^{1/6}} e^{-\alpha^2 T^{1/3}} \leq \frac{\bar{C}}{\alpha^{6} T}.
\label{eq_1226}
\end{equation}
The lemma follows from (\ref{eq_1059}) and (\ref{eq_1226}).
\end{proof}

\begin{proof}[Proof of Theorem~ \ref{thm_1328}] Set $R = \cN(\theta)$. We assume that $R \delta^4 \leq n$, since otherwise the conclusion of the theorem is vacuous. According to Lemma~\ref{lem_dom1},
\[ \int_{-[n / (R \delta^4)]^{1/6}}^{[n / (R \delta^4)]^{1/6}}
    \left| \vphi_{\theta}(\xi) - e^{-\xi^2/2} \right| \frac{d \xi}{|\xi|}
    \leq \frac{C R}{n} \left[ \gamma^3 + \delta^4 \right] \leq \frac{\tilde{C} R \delta^4}{n}~, \]
for $\gamma^3 = \EE X^3 \leq \delta^2$. It still remains to bound the integral for
$[n / (R \delta^4)]^{1/6} \leq |\xi| \leq n / (R \delta^4)$.
 To that end, we use Lemma~\ref{lem_1101}, which states that,
$$ |\vphi_{\theta}(\xi)| \leq \exp \left \{ -4 S^2(\xi) \right \} \quad \quad (\xi \in \RR). $$
According to Lemma \ref{lem_1055}, we may apply Lemma \ref{lem_1103} for the function $S(\xi)$, with
$T = n / (R \delta^4)$ and with $\alpha$ being a universal constant. We deduce that
   \begin{eqnarray*} \lefteqn{ \int_{[n / (R \delta^4)]^{1/6} \leq |\xi| \leq n / (R \delta^4)}
      \left| \vphi_{\theta}(\xi) - e^{-\xi^2/2} \right| \frac{d
    \xi}{|\xi|}}  \\ & \leq & C \exp \left \{ -c \left[ n / (R \delta^4) \right]^{1/3} \right \} +
    \int_{[n / (R \delta^4)]^{1/6} \leq |\xi| \leq n / (R \delta^4)} \exp \left\{ -4
S^2(\xi) \right\}
\frac{d\xi}{|\xi|}  \leq \frac{\tilde{C}  \delta^4 R}{n}.
\end{eqnarray*}
The theorem now follows from (\ref{eq_958}).
\end{proof}

Let us verify that $\cN(\theta^0) \leq C$ for a universal constant
$C > 0$, where $\theta^0$ is the unit vector defined in (\ref{eq_1412_}).

\begin{lemma} For any $\xi \in \RR$,
$$
 d(\xi \theta^0, \ZZ^n) \geq \min \left \{ |\xi|,  \frac{c n}{|\xi|} \right \}, $$
where $c >
0$ is a universal constant. \label{l:sqrt2}
\end{lemma}

\begin{proof} Liouville's theorem (see, e.g., \cite[Section II.6]{CRS}) states that for any integer $p, q
\neq 0$,
$$  \left|\sqrt{2} - \frac{q}{p} \right| \geq \frac{c_1}{p^2}, $$
for some universal constant $c_1 > 0$. Let $|\xi| > 1/2$ and suppose
that $p, q \in \ZZ$ are integers that satisfy $|\xi - p| = d(\xi,
\ZZ)$ and $|\xi \sqrt{2} - q| = d(\xi \sqrt{2}, \ZZ)$. Then,
\[ \frac{c}{|\xi|} \leq \frac{c_1}{ |p| } \leq \left|p \sqrt{2} - q\right| \leq |p \sqrt{2} - \xi \sqrt{2}| + |\xi \sqrt{2} - q|  = \sqrt{2} d(\xi, \ZZ) + d(\xi \sqrt{2}, \ZZ)~.  \]
We deduce that for any $\xi \in \RR$,
$$ d^2 \left( \xi , \ZZ \right) + d^2 \left( \xi \sqrt{2} , \ZZ
\right) \geq \min \{3  \xi^2, \tilde{c} \xi^{-2} \}. $$ According to the
definition of the unit vector $\theta^0$, and see that for $\xi \in \RR$,
$$
 d^2(\xi \theta^0, \ZZ^n) =  \frac{n}{2} \left[ d^2 \left(
\sqrt{\frac{2}{3 n}} \xi , \ZZ \right) + d^2 \left( \sqrt{\frac{2}{3
n}} \xi \sqrt{2}, \ZZ \right) \right]  \geq \min \{ \xi^2, c n^2 /
\xi^2 \}~.
$$

\end{proof}

{\small

\bigskip
{\small \noindent  School of Mathematical Sciences, Tel-Aviv
University, Tel-Aviv 69978, Israel

{\small \noindent \it e-mail address:} {\small
\verb"[klartagb,sodinale]@tau.ac.il"}


\begin{thebibliography}{99}

\bibitem{AL} Adamczak, R., Litvak, A. E. , Pajor, A., Tomczak-Jaegermann, N.,
{\it Restricted isometry property of matrices with independent
columns and neighborly polytopes by random sampling}. Preprint.
Available under  \verb"http://arxiv.org/abs/0904.4723"


\bibitem{BCLM}
Barthe, F., Cordero-Erausquin, D., Ledoux, M., Maurey, M.,
{\it Correlation and Brascamp-Lieb inequalities
for Markov semigroups}. Preprint. Available under \\
\verb"http://arxiv.org/abs/0907.2858"

\bibitem{CLL} Carlen, E., Lieb, E., Loss. M.,
{\it A sharp analog of Young's inequality on $S^N$ and related
entropy inequalities.} J. Geom. Anal., 14, (2004), 487--520.

\bibitem{CRS} Courant, R., Robbins, R., Stewart, I.,
{\it What Is Mathematics? An Elementary Approach to Ideas and
Methods}. Oxford University Press, 1996.

\bibitem{DF} Diaconis, P., Freedman, D., {\it
Asymptotics of graphical projection pursuit.} Ann. Statist., 12, no.
3, (1984), 793--815.

\bibitem{DF2} Diaconis, P., Freedman, D., {\it
A dozen de Finetti-style results in search of a theory.} Ann.
    Inst. H. Poincar\'e Probab. Statist., 23, no. 2, (1987), 397--423.


\bibitem{feller} Feller, W.,
{\it An introduction to probability theory and its applications,
volume I+II.} John Wiley \& Sons, Inc., New York-London-Sydney,
1971.

\bibitem{FS} Friedland, O., Sodin, S., {\it Bounds on the concentration function in terms of the
Diophantine approximation. } C. R. Math. Acad. Sci. Paris, 345, no.
9, (2007),  513--518.

\bibitem{H} Hitczenko, P., Montgomery-Smith, S. J., Oleszkiewicz, K.,
{\it Moment inequalities for sums of certain independent symmetric
random variables. } Studia Math., 123, no. 1, (1997), 15--42.

\bibitem{IL} Ibragimov, I., Linnik, Yu. V., {\it Independent and stationary
sequences of random variables.} Translation from the Russian, edited
by J. F. C. Kingman. Wolters-Noordhoff Publishing, Groningen, 1971.

\bibitem{Kahane} Kahane, J. P., {\it Some random series of functions.}
 Second edition. Cambridge Studies in Advanced Mathematics, 5.
 Cambridge University Press, Cambridge, 1985.

\bibitem{ptrf} Klartag, B., {\it
A Berry-Esseen type inequality for convex
bodies with an unconditional basis.}
Probab. Theory Related Fields, vol. 45, no. 1, (2009), 1--33.



\bibitem{RV} Rudelson, M., Vershynin, R., {\it
The Littlewood-Offord problem and invertibility of random matrices}.
Adv. Math. 218, no. 2, (2008), 600--633.

\bibitem{S} Sudakov, V. N., {\it Typical distributions of linear functionals in finite-dimensional spaces
of high-dimension.} (Russian) Dokl. Akad. Nauk. SSSR, 243, no. 6,
(1978), 1402--1405. English translation in Soviet Math. Dokl., 19,
(1978), 1578--1582.


\end{thebibliography}
\end{document}